\documentclass[11pt]{amsart}
\usepackage{geometry}
\usepackage{amssymb}
\numberwithin{equation}{section}

\newtheorem{theorem}{Theorem}[section]
\newtheorem{lem}[theorem]{Lemma}
\newtheorem{prop}[theorem]{Proposition}
\newtheorem{defi}[theorem]{Definition}

\newtheorem{exa}[theorem]{Example}
\newtheorem{remark}[theorem]{Remark}

\newtheorem*{ackn}{Acknowledgements}

\newtheorem*{thmM}{Main Theorem} 

\newcommand{\PSH}{\textup{PSH}}
\newcommand{\R}{\mathbb R}
\newcommand{\N}{\mathbb N}
\newcommand{\f}{\varphi}
\newcommand{\p}{\psi}
\newcommand \e {\varepsilon}
\newcommand \la {\lambda}
\newcommand{\tr}{\operatorname{tr}}

\usepackage{xcolor}
\usepackage{hyperref}
\hypersetup{   
colorlinks=true,       
linkcolor=black,          
citecolor=black,        
filecolor=black,      
urlcolor=black}

\begin{document}

\title{K\"ahler-Ricci Flow: from divisors to cusps}

\date{\today}

\author{Eleonora Di Nezza}

\address{Universit\`a di Roma TorVergata, Dipartimento di Matematica Via della Ricerca Scientifica 1, 00133 Roma, Italy.} 
\email{\href{mailto:dinezza@mat.uniroma2.it}{dinezza@mat.uniroma2.it}}

\author{Vincent Guedj}
\address{Institut Universitaire de France \& Institut de Mathématiques de Toulouse\\
  Université de Toulouse; CNRS, UPS\\
  118 route de Narbonne, F-31400 Toulouse\\
  France}
  \email{\href{vincent.guedj@math.univ-toulouse.fr}{vincent.guedj@math.univ-toulouse.fr}}

\author{Chinh H. Lu}
\address{Institut Universitaire de France \& Université d'Angers, CNRS, LAREMA, SFR MATHSTIC, F-49000 Angers, France}
\email{\href{hoangchinh.lu@univ-angers.fr}{hoangchinh.lu@univ-angers.fr}} 

 \begin{abstract}
We study the geometric regularization of positive closed currents by the K\"ahler-Ricci flow
on compact K\"ahler manifolds. In a previous work of ours, it was shown that the K\"ahler-Ricci flow 
immediately smoothes out such a current when it has zero Lelong numbers. 
We study here the case when $T_0$ has divisorial singularities, showing that the 
flow gradually replaces the latter by Poincaré type ones, providing 
an approximation of $T_0$ by complete K\"ahler metrics with bounded curvature in a Zariski open set.
 \end{abstract}

\maketitle

\setcounter{tocdepth}{1}
\tableofcontents

\section*{Introduction}

Smoothing properties of the Ricci flow have been known and used for a long time in K\"ahler geometry
(see e.g. \cite{BM87,Tian97,PSSW08,SzTo11,ST17}). It has been notably shown in \cite{GZ17,DNL17}
that one can run the (twisted) K\"ahler-Ricci flow from a very degenerate initial datum, 
namely an arbitrary positive  closed current of bidegree $(1,1)$.

By comparison with the riemannian setting, the K\"ahler analysis is simplified by the fact that
the problem at hand is equivalent to a scalar parabolic equation, a complex Monge-Amp\`ere flow.
Let $(X,\omega)$ be a compact K\"ahler manifold of dimension $n$.
We fix $\f_0$ an $\omega$-plurisubharmonic function and
consider the complex Monge-Ampère flow 
\[
(\omega+dd^c \varphi_t)^n = e^{\dot \varphi_t}dV_X,
\]
starting from $\f_0$, where $dV_X$ is a fixed smooth volume form on $X$.
Here  $d=\partial+\overline{\partial}$, and $d^c=\frac{1}{2i \pi} (\partial-\overline{\partial})$,
so that $dd^c =\frac{i}{ \pi} \partial\overline{\partial}$.

\smallskip

The function $\varphi_t$ is the "maximal weak solution" to the flow.
As shown in \cite{GZ17,DNL17}, it is smooth in a Zariski open set $\Omega$ where it solves the equation in the classical sense,
and it has logarithmic singularities at the boundary of this set.
Among all possible solutions to the flow in $\Omega$, $\f_t$ is the unique one having minimal singularities
along $\partial \Omega$.

Equivalently $\omega_t=\omega+dd^c \f_t$ is a K\"ahler form in $\Omega$
which is a solution of the twisted K\"ahler-Ricci flow
\begin{equation*} \label{eq:KR}
\tag{KR}
\frac{\partial \omega_t}{\partial t}=-{\rm Ric}(\omega_t)+{\rm Ric}(dV_X),
\end{equation*}
and whose potential $\f_t$ has the smallest singularities along $\partial \Omega$.

Our main goal in this article  is to give a precise description of the singularities before the smoothing time.
We will show that divisorial singularities are gradually replaced by Poincaré type ones.
This disproves  a folklore conjecture in the field that the K\"ahler-Ricci flow produces approximants with analytic singularities 
(see e.g. \cite[Question 43]{DGZ16}).

\smallskip

Demailly has produced several regularization results of quasi-plurisubharmonic functions, using either 
convolutions \cite{Dem92} or approximations by Bergman kernels \cite{DPS01}.
While Demailly's Bergman kernel approach naturally generates approximants with analytic singularities 
for $\omega$-plurisubharmonic functions,  
the K\"ahler-Ricci flow smoothing exhibits a more geometric character,
as shown by our main result:
 
 \begin{thmM} \label{thm:A}
 Let $(X,\omega)$ be a compact K\"ahler manifold of dimension $n$.
 Let $D=\sum_{j=1}^r D_j$ be an ample log-smooth divisor, i.e. a finite 
convex combination of irreducible smooth ample hypersurfaces with simple normal crossings.
Fix $m_j \geq 0$, $s_j$ holomorphic sections of ${\mathcal O}(D_j)$,
$h_j$ hermitian metrics of ${\mathcal O}(D_j)$. Finally set $\f_0=\sum_{j=1}^r m_j \log |s_j|^2$
and $T_0=\omega+dd^c \f_0 \geq 0$.

The solution $(\omega_t)_{t>0}$ of the twisted K\"ahler-Ricci flow  \eqref{eq:KR} emanating from $T_0$
satisfies:
   \begin{itemize}
   \item    $\omega_t= \sum_{j=1}^r \max(m_j-t,0)  [D_j] + 	\beta_{t,D}$,
   where $\beta_{t,D}$ is a  Poincaré metric in $X \setminus \cup_{m_j \geq t} D_j$;
   \item   $\omega_t$ is  a K\"ahler form on the whole of $X$ as soon as $t>\la(\f_0)=\max_j m_j$.
   \end{itemize}
   In particular, $\omega_t$ has bounded curvature in $X \setminus \cup_{m_j \geq t} D_j$.
 \end{thmM}

Here $|s_j|^2$ is the norm of $s_j$  with respect to the metric $h_j$. 
The current of integration along $D_j$ satisfies
$[D_j]=\Theta_{h_j}+dd^c \log |s_j|^2$, hence we implicitly assume that
$\omega \geq \sum_{j=1}^r m_j \Theta_{h_j}$.
The solution of the flow is explicit when $\omega = \sum_{j=1}^r m_j \Theta_{h_j}$ 
(see Proposition \ref{pro:decompo}).


The proof of the main theorem is simpler when $m_1=\cdots=m_r$.
When the multiplicities $m_j$ do not coincide, the analysis requires one to subdivide
the smoothing interval $[0,\la(\f_0)]$, and construct a Poincaré metric 
$\beta_{t,D}$ which is smooth near the divisors $\cup_{m_j < t} D_j$.

\smallskip
 
 In {\it Section \ref{sec:maxflow}}
 after recalling the construction of the flow,
 we  establish a general comparison principle (Theorem \ref{thm:pcp}),
 as well as the uniform quasi-concavity in time of the potentials (Proposition \ref{pro:concave}).
  The striking uniform lower bound 
 $\dot{\f}_t \geq n \log t -C$ follows then easily.
  We then study the zero order asymptotics of $\f_t$ in  {\it Section \ref{sec:uniform}}.
 We first analyze the divisorial singularities in Lemma \ref{lem:log}, and then establish 
 uniform estimates along the flow in Proposition \ref{prop:0uniform1}, and 
 in Proposition \ref{prop:0uniform2} when the $m_j$'s do not coincide.
  We finally prove our main result in {\it Section \ref{sec:logsmooth}}, first 
 refining the control on $\dot{\f}_t$ in Proposition \ref{pro:STbelow}, then
 obtaining a uniform ${\mathcal C}^2$-estimate in space (Theorem \ref{thm:c2hyp}).
Higher order estimates and
the comparison with Poincaré metrics are finally obtained in Theorem \ref{thm:main},
by using Evans-Krylov theory, Schauder estimates, and the quasi-coordinates adapted to Poincaré metrics.

\begin{ackn} 
We thank S.Boucksom and H.Guenancia for insightful discussions.
The authors are partially supported by the Institut Universitaire de France
and the fondation Charles Defforey. 
E.D.N. is partially supported by the E.R.C. grant SiGMA, No. 101125012;
V.G. is partially supported by the Clay foundation. C.H.L is partially supported by the Centre Henri Lebesgue ANR-11-LABX-0020-01.
This material is based upon work supported by the National Science Foundation under Grant No. DMS-1928930, while the authors were in residence at the Simons Laufer Mathematical Sciences Institute in Berkeley, California, Fall 2024,
as part of the Special Geometric Structures and Analysis program.  
\end{ackn}

\section{The maximal flow} \label{sec:maxflow}

In the whole article we let $(X,\omega_X)$ be a compact K\"ahler manifold of dimension $n$. 
In this section we let $\omega$ be a closed smooth $(1,1)$-form that is semipositive and big,
i.e. with positive volume $\int_X \omega^n>0$ (see Definition \ref{defi:big-ample locus}).

We recall the construction of the maximal solution to the K\"ahler-Ricci flow \eqref{eq:KR}
following \cite{GZ17,DNL17,Dan25} and we establish new general and important estimates,
notably Theorem \ref{thm:pcp} and Proposition \ref{pro:concave}, which will be useful
both in this article when $\omega$ is K\"ahler, and in \cite{DGL26} when 
$\omega$ is merely semi-positive and big.

\subsection{Quasi-plurisubharmonic functions}

\subsubsection*{Analytic singularities}

  A function is quasi-plurisub\-harmonic (quasi-psh) if it is locally given as the sum of  a smooth and a plurisubharmonic function.   
 
\begin{defi}
Quasi-psh functions
$\f:X \rightarrow \R \cup \{-\infty\}$ satisfying
$\omega+dd^c \f \geq 0$
in the weak sense of currents are called $\omega$-plurisubharmonic ($\omega$-psh for short).
We let $\PSH(X,\omega)$ denote the set of all $\omega$-plurisubharmonic functions which are not identically $-\infty$.  
\end{defi}

Note that constant functions are $\omega$-psh functions.
A ${\mathcal C}^2$-smooth function $u$ has bounded Hessian, hence $\e u$ is
$\omega_X$-psh if $0<\e$ is small enough.

\begin{defi}\label{defi:analsing} 
A quasi-psh function $\f$ has \emph{analytic singularities}
 if it can be locally written as 
$
\f=\frac{1}{2m}\log\sum_{j=1}^N|f_j|^2+u
$
for some holomorphic functions $f_1,\dots,f_N$, $m\in\N^*$ and smooth function $u$. 
If we can further take $N=1$ then we say that $\f$ has \emph{divisorial singularities}.
\end{defi}

\begin{defi}\label{defi:big-ample locus} 
We say that $\omega$ is \emph{big} if there exists $\varepsilon>0$ and
an $\omega$-psh function $u$ such that $\omega+dd^c u \geq \varepsilon \omega_X$.
The current $\omega+dd^c u$ is called a \emph{K\"ahler current}.

The \emph{ample locus} $\mathrm{Amp}(\omega)$  is the Zariski open subset 
of points $x\in X$ such that there exists a K\"ahler current 
$\omega+dd^c u$
which is smooth in a neighborhood of $x$. 
\end{defi}

The notion of analytic singularities  depends on the choice of local holomorphic functions, 
not only on the ideal sheaf they generate.
We recall the following \cite[Lemma 1.7]{BGL25}:

\begin{lem}\label{lem:analres} 
If $\f$ is a quasi-psh function with analytic singularities, there exists a modification $\pi\colon Y\to X$, isomorphic over 
 $\{\f>-\infty\}$, such that $\pi^\star\f$ has divisorial singularities. 
\end{lem}

\subsubsection*{Size of singularities}

\begin{defi}
Given a quasi-psh function $\f$ on $X$, we let 
$$
c(\f):=\sup \left\{c>0, \int_X e^{-c\f} dV_X <+\infty \right\}
$$
denote the integrability index of $\f$, while $\la(\f)=c(\f)^{-1}$ is the Arnold multiplicity of $\f$.

We define similarly $c(\f,x)$ and $\la(\f,x)=c(\f,x)^{-1}$ the local version of these invariants,
and let $\nu(\f,x)=\sup \{ \gamma \geq 0, \; \f(y) \leq \gamma \log {\rm d}_{\omega}(x,y)+C \}$
be the Lelong number of $\f$ at $x \in X$.
\end{defi}

It follows from Skoda's integrability theorem \cite[Theorem 2.50]{GZbook} that for all $x \in X$,
$
\frac{\nu(\f,x)}{n} \leq \la(\f,x) \leq \nu(\f,x).
$
We refer the reader to \cite{DK01,GZbook} for basic properties of these numerical invariants, 
as well as their counterpart in algebraic geometry. We  stress here that
$$
\la(\f)=\sup \{ \la(\f,x), \; x \in X \},
$$
while $c(\f)=\inf \{ c(\f,x), \; x \in X \}$,
and that $(\f,x) \mapsto \nu(\f,x)$  
and $(\f,x) \mapsto \la(\f,x)$ are upper semi-continuous (see\cite[Main theorem]{DK01}).

\begin{theorem} \label{thm:skoda}
There exists $\alpha>0$ and $C_{\alpha} >0$ such that for all $\f \in PSH_0(X,\omega)$,
$$
\int_X e^{-\alpha \f} dV_X \leq C_{\alpha}.
$$
\end{theorem}

The exponent $\alpha$ only depends on an upper bound on the cohomology class $\{\omega\}$, while
the constant $C_{\alpha}$ depends on $\alpha$ and $(X,\omega)$ in a controlled way
(see \cite[Theorem 2.9]{DGG23}).

\subsubsection*{Demailly's approximation}

The importance of functions with analytic singularities stems from the following consequence of 
Demailly's regularization theorem \cite{DPS01}: 

\begin{theorem} \label{thm:Demreg} 
Any $\f\in\PSH(X,\omega)$ can be written as the limit of a decreasing sequence $\f_j\in\PSH(X,\omega+2^{-j}\omega_X)$ with analytic singularities, 
such that $c(\f)=\lim c(\f_j)$. In particular, when $\omega$ is K\"ahler, the sequence $(\varphi_j)$ can be taken in $\PSH(X,\omega)$. 
\end{theorem}

It is easy to smooth out an $\omega$-psh function with analytic singularities, in particular
any such function can be written as the limit of a decreasing sequence of smooth $\omega$-psh functions.

\smallskip

The set $\PSH(X,\omega)$ is a closed subset of $L^1(X)$.
Subsets of  $\omega$-psh functions enjoy strong compactness  properties:
$
\PSH_A(X,\omega):=\{ u \in\PSH(X,\omega), \, -A \leq \sup_X u \leq 0 \} 
$
is a compact subset of $L^1(X)$ for any $A>0$.
We refer the reader to \cite{GZbook} for further basic properties of $\omega$-psh functions.

\subsubsection*{Finite energy class}

Fix $\f \in PSH(X,\omega)$ and set $\f_j:=\max(\f,-j) \in PSH(X,\omega) \cap L^{\infty}(X)$.
It has been shown in \cite{GZ07} that the measures ${\bf 1}_{\{\f<-j\}} (\omega+dd^c \f_j)^n$
increase, as $j$ increases to $+\infty$, to a positive Radon measure 
$\mu_{\f}$ with total mass $\mu_{\f}(X) \leq \int_X \omega^n$.

\begin{defi}
The class $ \mathcal{E}(X,\omega)$ is the set of all $\omega$-psh functions with full Monge-Amp\`ere mass,
i.e. such that $\mu_{\f}(X) =\int_X \omega^n$.
\end{defi}

Functions in the finite energy class $\mathcal{E}(X,\omega)$ can be unbounded, but they have mild singularities, in particular
their Lelong numbers vanish identically (see \cite[Theorem 1.1]{DDL2}).

\subsection{Basic properties of the flow} \label{sec:construction}

\subsubsection*{Construction} 
Fix $\f_0 \in \PSH(X,\omega)$ and set $\Omega_t=\{x \in X, \nu(\f_0,x)<t \}$.
The latter form an increasing sequence of Zariski open subsets such that
$\Omega_T=X$ as soon as $T>T_0=\max_{x \in X} \nu(\f_0,x)$.
When $\omega$ is K\"ahler, by Theorem \ref{thm:Demreg} we can approximate $\f_0$ by a decreasing
sequence $(\f_{0,j})$ of smooth $\omega$-psh functions.
Let $\f_{t,j}$ denote the unique smooth solutions of 
$$
(\omega+dd^c \varphi_{t,j})^n = e^{\dot \varphi_{t,j}}dV_X,
$$
starting from   $\f_{0,j}$. It is shown in \cite{GZ17,DNL17} that
for all $(t,x) \in \R^+ \times X$, 
\begin{itemize}
\item the sequence  $j \mapsto \f_{t,j}(x)$ decreases to $\f_t(x) \in \R \cup \{-\infty\}$;
\item the limit $\f_t(x)$ is independent of the choice of approximants;
\item the function $x \mapsto \f_t(x)$ belongs to $PSH(X,\omega)$;
\item the function $(t,x) \mapsto \f_t(x)$ is smooth in $\R_*^+ \times \Omega_s$, and
$$
(\omega+dd^c \varphi_{t})^n = e^{\dot \varphi_{t}}dV_X
\; \; \text{ in } \; \;
\Omega_s, \; s<t;
$$
\item the function $\f_t$ may have positive Lelong numbers at some points in $\partial \Omega_t$;
\item one has $\f_t \rightarrow \f_0$ in $L^1$ as $t \rightarrow 0$,
and the convergence holds in the sense of capacity.
\end{itemize}
Moreover,  the flow becomes completely smooth  in finite time (for $t>T_0$).

\medskip

When $\omega$ is merely semipositive and big, one can approximate $\f_0$ by a decreasing sequence $(\f_{0,j})$ of 
smooth $(\omega+2^{-j} \omega_X)$-psh functions (by Theorem \ref{thm:Demreg}) 
and establish similar results in the ample locus of $\omega$ (see \cite{Dan25}).

\subsubsection*{Sub/super-solutions}

\begin{defi}
An upper semicontinuous function $\f: \mathbb R^+ \times X \rightarrow \R \cup \{-\infty\}$ is called a subsolution  of the flow if 
\begin{itemize}
\item $\f_t(\cdot)=\f(t,\cdot)$ belongs to $PSH(X,\omega)$ for all $t \in \R^+$; 
\item $(t,x) \mapsto \f_t(x)$ is smooth in $\R^+_* \times \Omega$ for
some Zariski dense open set $\Omega \subset X$, and
$$
(\omega+dd^c \varphi_{t})^n \geq e^{\dot \varphi_{t}}dV_X
\; \; \text{ in } \; \;
\Omega.
$$
\end{itemize}
\end{defi}

One defines similarly the notion of supersolution, and the notion of (weak) solution: 

 \begin{defi}
A function $\f: \R \times X  \rightarrow \R \cup \{-\infty\}$ is  a supersolution of the flow if 
 $(t,x) \mapsto \f_t(x)$ is smooth in $\R^+_* \times \Omega$ for
some Zariski dense open set $\Omega \subset X$, and for  $t>0$,
$$
{\bf 1}_{\{\omega +dd^c \varphi_t \geq 0\}} (\omega+dd^c \varphi_{t})^n \leq e^{\dot \varphi_{t}}dV_X
\; \; \text{ in } \; \;
\Omega.
$$

A function  is   a weak solution of the flow
if it is both a subsolution and a supersolution.
\end{defi}

Note that supersolutions are {\it not} required to be a family of $\omega$-psh functions.
Let us also stress that these notions of (sub/super)solutions differ significantly
from the ones studied in \cite{GLZ20,GLZ21}. In the latter works the functions under consideration
are globally bounded, so that the Monge-Amp\`ere measures 
$(\omega+dd^c \varphi_{t})^n$ have no mass at the boundary of $\partial \Omega$,
as follows from the Chern-Levine-Nirenberg inequality \cite[Theorem 3.14]{GZbook}.

\smallskip

The crucial problem for us here is to understand the precise asymptotic behavior 
of the functions $\f_t$ along $\partial\Omega$.
We provide below examples of different weak solutions to the flow having various types of
asymptotic along $\Omega$ (see Example \ref{exa:notunique}). 
An important role is therefore played by the following property (see \cite[Theorem 5.8]{DNL17} and \cite[Theorem B]{Dan25}):

\begin{theorem} \label{thm:maxflow}
The solution 
 constructed in Section \ref{sec:construction} is the envelope of
all subsolutions to the flow. In particular it has minimal singularities among all
weak solutions to the flow.
\end{theorem}

In the rest of this article we call this solution {\it the} solution of the flow.
A simple consequence of this maximality property, 
not explicitly stated in \cite{GZ17,DNL17}, is the following:

\begin{lem} \label{lem:semigroup}
Fix $s>0$.
Let $\f_t$ be the solution of the flow starting from $\f_0$,
and let $\p_t$ denote the solution of the flow starting from  $\f_s$.
Then $\p_t=\f_{t+s}$ for all $t \geq 0$.
\end{lem}

This property  follows from  the invariance under translations in time: 
if $(t,x) \mapsto \f(t,x)$ is a (sub)solution of the equation then, 
 for $s > 0$, so is the function $(t,x) \mapsto \p(t,x) = \f(t + s,x)$.

\subsection{Comparison principle}
Recall that $\omega$ is semipositive and big. 
Fix $\Omega \subset {\rm Amp}(\omega)$ a Zariski open set,
and $\rho$ a strictly $\omega$-psh  function which is smooth in $\Omega$
 with $\partial \Omega = \{\rho=-\infty\}$.
The following result is an important variation on the classical parabolic maximum principle:

\begin{theorem}\label{thm:pcp}
Let $\varphi_t$ (resp. $\psi_t$) be a subsolution (resp. supersolution) of the flow 
\[
(\omega+dd^c u_t)^n =e^{\dot{u}_t +f} dV_X, 
\]
 with 
$\varphi,\psi,f$ smooth in $(0,T] \times  \Omega$. 
Assume the following properties hold:
\begin{enumerate}
\item $\psi$ is continuous on $[0,T]\times \Omega$ and $\varphi_0\leq \psi_0$;
\item there exists $h\in \mathcal{E}(X,\omega)$ such that $\varphi_t(x)+h(x)\leq \psi_t(x)$, 
for all $(t,x) \in [0,T] \times \Omega$;
\item $\rho(x) \leq \psi_t(x)$, for all $(t,x) \in [0,T] \times \Omega$.
\end{enumerate}
Then $\f_t(x) \leq \p_t(x)$ for all $(t,x) \in [0,T] \times \Omega$.
\end{theorem}

In what follows we often work with sub/supersolutions whose dominant singular part is of the form $(1-t)\log |s|^2$, where $s$ is a holomorphic section.
 In this setting, we apply the above comparison principle to $\psi_{t+\varepsilon}$ 
 for which the condition 3) is satisfied and let then $\varepsilon\to 0$.  

\smallskip

Let us stress that we allow the density $f$ to take values $\pm \infty$ at some points of $\partial \Omega$; this will be  quite useful in later applications.

\begin{proof}
If the maximum of the function $\varphi_t-\psi_t-2Ct\varepsilon$ is attained in $\Omega$,  the classical maximum principle provides the desired result, as the functions are smooth in $\Omega$. In order to force this maximum to be attained in $\Omega$, we need to add $h$ and $\rho$ to $\varphi_t$.  To deal with 
the lack of regularity of $h$, we use the  viscosity theory.

We assume without loss of generality that $\omega +3dd^c \rho\geq a\omega_X$, for some $a>0$. Fix $C>0$ so large that $(\omega/2+dd^c \rho)^n \geq e^{-C}dV_X$ in $\Omega$. 
Thanks to \cite[Proposition 2.5]{ALS24} we have $P_\omega (4h)\in \mathcal{E}(X,\omega)$,
thus up to working with $\frac{1}{4}P_\omega (4h) $ we can assume that $h\in \mathcal{E}(X, \omega/4)$. Here $P_\omega (4h)$ denotes the largest $\omega$-psh function that lies below $4h$. 
We set $\omega'=\omega/4$.


We claim that we can find a concave increasing function $\chi : \mathbb R^- \rightarrow \mathbb R^-$ with 
$\frac{\chi(t)}{t} \to +\infty$ as $t\to -\infty$ such that $P_{\omega'}(\chi(h)) \in \mathcal{E}(X,\omega')$.
Indeed since $h \in {\mathcal E}(X,\omega')$ there exists a convex increasing function 
$\chi_1: \mathbb R^- \rightarrow \mathbb R^-$ with $\chi_1(-\infty)=-\infty$, $\chi_1(h)\in L^1(X,(\omega'+dd^c h)^n)$
(see \cite[Proposition 2.2]{GZ07}), and moreover
$
\int_X |\chi_1(h_j)| (\omega' +dd^c h_j)^n \leq C_1,
$
for some  $C_1>0$, where $h_j=\max(h,-j)$. 
We set $\chi'(t)=(-\chi_1(t))^{1/2n}$ and $\gamma(t)= -(-\chi_1(\chi^{-1}(t)))^{1/2}$,
with $\chi(0)=0$. Note that $\gamma$ is  increasing  with $\gamma(-\infty)=-\infty$. 
Since the Monge-Amp\`ere measure of $u_j= P_{\omega' }(\chi(h_j))$ is concentrated on 
 $D=\{P_{\omega'}(\chi(h_j))= \chi(h_j)\}$, we obtain
$$
{\omega'}_{u_j}^n = {\bf 1}_D (\omega'+dd^c \chi(h_j))^n 
\leq  (\chi'(h_j))^n {\omega'}_{h_j}^n \leq (-\chi_1(h_j))^{1/2}  {\omega'}_{h_j}^n.
$$ 
Thus 
$
 \int_X |\gamma(u_j)|(\omega'+dd^c u_j)^n \leq \int_X |\chi_1(h_j)| (\omega' +dd^c h_j)^n \leq C_1
$
hence $P_{\omega'}( \chi(h)) \in \mathcal{E} (X, \omega')$.

\smallskip

The function  $v=P_{\omega/2}(\rho +\chi(h))$ is not identically $-\infty$ since it dominates 
$\rho + P_{\omega/4} (\chi(h))$ which is $\omega/2$-psh. 
Fix $\varepsilon\in (0,1)$ and consider the function 
\[
\varphi_{t,\varepsilon} =  (1-\varepsilon) \varphi_t+\varepsilon \rho-2C\varepsilon t. 
\]
The concavity of $H \mapsto \log \det H$ on the set of positive hermitian matrices 
ensures that if $\alpha,\beta >0$ are positive $(1,1)$-forms such that
$\alpha^n \geq f dV$ and $\beta^n \geq g dV$, then 
$$
\alpha^j \wedge \beta^{n-j} \geq f^{\frac{j}{n}} g^{\frac{n-j}{n}} dV_X
$$
for all $0 \leq j \leq n$. Using the binomial expansion we infer for all $0<\e<1$,
$$
((1-\e) \alpha+\e \beta)^n \geq \left( (1-\e) f^{\frac{1}{n}} +\e g^{\frac{1}{n}} \right)^n dV_X.
$$
Using this inequality and the convexity of  the exponential, we obtain
\begin{eqnarray*}
((1-\varepsilon/2)\omega + dd^c \varphi_{t,\varepsilon})^n  
&\geq & \left((1-\varepsilon) e^{\dot \varphi_t/n} + \varepsilon  e^{-C/n}\right)^ndV_X\\
&\geq &  e^{(1-\varepsilon) \dot \varphi_t -C\varepsilon }dV_X= e^{\dot{\varphi}_{t,\varepsilon}+C\varepsilon}dV_X.
\end{eqnarray*}

The function $(t,x) \mapsto u(t,x):= \varphi_{t,\varepsilon}+ \varepsilon v -\psi_{t}$
 is upper semicontinuous in $[0,T]\times \Omega$ (since $\psi_t$ is continuous).
 The assumptions (2), (3) ensure that $u\leq \varepsilon v-(1-\varepsilon)h \leq \varepsilon \rho + \varepsilon \chi(h)-(1-\varepsilon)h $. The growth condition on $\chi$ ensures that $u$ reaches its maximum in $[0,T]\times \Omega$.
  We claim that the latter is moreover reached
 along $(t=0)$, so that 
 $$
 u \leq \sup_{\Omega} (\f_{0,\e}+\varepsilon v -\psi_0) \leq \varepsilon \sup_{\Omega} (\rho-\psi_0) \leq 0,
 $$
using $\varphi_0\leq \psi_0$.
 Letting $\varepsilon \to 0^+$ gives the conclusion. 
 
 \smallskip
 
  Assume by contradiction that $u$ reaches its maximum at some $(t_0,x_0)\in [0,T]\times \Omega$, with  $t_0>0$. 
 At the point  $(t_0,x_0)$, the classical maximum principle yields
\[
\partial_t \varphi_{t,\varepsilon}  -\partial_t \psi_{t}=  \partial_t u \geq 0. 
\]
Take a local smooth potential $g$ of $\omega/2$ around $x_0$. Since $v$ is $\omega/2$-psh, by \cite[Proposition 1.3]{EGZ11}, $g+v$ is a viscosity subsolution of the equation $(dd^c \cdot )^n \geq 0$ near $x_0$. 
Also, the function $(\psi_{t_0} +\varepsilon g-\varphi_{t_0,\varepsilon})$ is smooth near $x_0$ and (by assumption) $\varepsilon (g+v)-(\psi_{t_0} +\varepsilon g-\varphi_{t_0,\varepsilon})$ attains its local maximum at $x_0$. It thus follows from the definition of viscosity subsolutions that 
$$
dd^c(\psi_{t_0} +\varepsilon g-\varphi_{t_0,\varepsilon})  \geq 0, 
$$
at $x_0$, which yields
$
0\leq \left(1-\frac{\varepsilon}{2}\right)\omega+dd^c \varphi_{t_0} \leq (\omega +dd^c \psi_{t_0}).
$
Since   $\psi$ is a   supersolution, we infer 
\[
e^{{\dot \psi}_{t_0}+f} dV_X \geq   (\omega +dd^c \psi_{t_0})^n \geq ((1-\varepsilon/2)\omega+dd^c \varphi_{t_0,\varepsilon})^n 
\geq e^{\dot{\varphi}_{t_0,\varepsilon}+C\varepsilon+f}dV_X \geq e^{{\dot \psi}_{t_0}+C\varepsilon+f} dV_X,
\]
a contradiction. Thus $t_0=0$.
\end{proof}

\begin{remark}
Prescribing the right asymptotic behavior near $\partial \Omega$ is crucial for this result to hold: we provide 
in Example \ref{exa:notunique}  several weak solutions  that do not have the same asymptotic. 
\end{remark}

\subsection{Quasi-concavity in time} \label{sec:concave}


\begin{prop} \label{pro:concave}
Let $\f_t$ be the solution of the flow $(\omega+dd^c \f_t)^n=e^{\dot{\f}_t} dV_X$ 
starting from  $\f_0 \in PSH(X,\omega)$.
The following estimates hold on $\R_+ \times X$:

\smallskip

\begin{enumerate}
\item $\f_0-Ct +n(t \log t-t) \leq \f_t \leq \sup_X \f_0 +Ct$;

\smallskip

\item $t \mapsto \f_t-n(t \log t-t)$ is concave;

\medskip

\item $n \log t -C \leq \dot{\f}_t\leq \frac{\f_t-\f_0}{t}+n$.
\end{enumerate}
Here $C>0$ is independent of $\f_0$ and only depends on $(X,\omega,dV_X)$.
\end{prop}

The bounds on $\dot{\f}_t$ largely generalize \cite[Lemma 3.2]{ST17} which dealt with {\it bounded} $\f_0$; 
the proof is  completely different and relies on the concavity property of $\phi_t=\f_t-n(t \log t-t)$.
By quasi-concavity the time derivative $\dot{\f}_t(x)$  is well defined off a countable set of times.

\smallskip

The function $\tilde{\phi_t}=\f_t-n(t \log t-t)+Ct$  is increasing in $t$.
This  property ensures that $\f_t$ converges to $\f_0$ in the strongest
possible sense as $t$ decreases to zero.

\begin{proof}
Fix $C>0$ such that $\omega^n \leq e^C dV_X$.
The upper bound $\f_t \leq \sup_X \f_0+Ct$ follows  from the maximum principle since $\sup_X \f_0+Ct$ is a super-solution.

It follows from Theorem \ref{thm:skoda} that we can find $0 < \tilde{\alpha} < \tilde{\alpha}(\{\omega\})$ 
such that $\int_X e^{-\tilde{\alpha} \f_0} \omega^n \leq C_{\tilde{\alpha}}$, with
$C_{\tilde{\alpha}}$ independent of $\f_0$. Set $\alpha= \frac{\tilde{\alpha}}{2}$. Using \cite[Corollary 11.9 and Theorem 12.1]{GZbook} 
we find a unique $u \in PSH(X,\omega) \cap L^{\infty}(X)$ such that
$||u||_{L^{\infty}(X)} \leq C_{\alpha}'$ and
$$
\alpha^n (\omega+dd^c u)^n=e^{\alpha (u-\f_0)} dV_X.
$$
Since  $(1-\alpha t) \f_0(x)+\alpha t u(x) +n(t\log t-t)$ is a subsolution of the flow,
we obtain
\begin{equation}\label{C0 esti}
\f_0-C_{\alpha}''t  +n(t\log  t-t) \leq
(1-\alpha t) \f_0+\alpha t u +n(t\log  t-t) \leq \f_t,
\end{equation}
for all $0 \leq t \leq 1/\alpha$, using that $\f_0 \leq 0$. 
Observe that $\f_{1/\alpha}$ is uniformly bounded,
$$
-M_{\alpha}=-C_{\alpha}'-\frac{n}{\alpha}(\log \alpha +1) \leq
u(x)-\frac{n}{\alpha}(\log \alpha +1) \leq \f_{1/\alpha} \leq 0.
$$
The semi-group property and the minimum principle ensure that 
$-M_{\alpha} \leq \f_t \leq 0$ for all $t \geq 1/ \alpha$.
Slightly enlarging $C_{\alpha}''$ if necessary, we thus obtain $(1)$.

\smallskip

To establish (2) one can proceed by approximation to reduce to the case when $\f_0,\f_t$ are smooth.
Observing that $\ddot{\f}_t=\Delta_{\omega_t}(\dot{\f}_t)$ one can differentiate once more and realize that
$$
\left( \frac{\partial}{\partial t}-\Delta_{\omega_t} \right)(t \ddot{\f}_t) \leq \ddot{\f_t} - nt \left( \frac{dd^c \dot{\varphi_t} \wedge \omega_t^{n-1}}{\omega_t^n}\right)^2
$$
(see computations in \cite[Theorem 2.7]{GLZ20}). The conclusion follows from the maximum principle. Indeed, consider $H(t,x)=t \ddot{\f}_t(x)$. If the maximum is attained at $(0, x_0)$ we then have $H(t,x)\leq H(0, x_0)=0$, hence $ \ddot{\f}_t \leq 0$, i.e. $\f_t $ is concave. In particular $\f_t-n(t\log t -t)$ is concave. If the maximum is attained at $t_0>0$, at $(t_0, x_0)$ we have that $$0\leq t\left( \frac{\partial}{\partial t}-\Delta_{\omega_t} \right)(t \ddot{\f}_t) \leq t\ddot{\f_t} - nt^2 \left( \frac{(dd^c \dot{\varphi_t}) \wedge \omega_t^{n-1}}{\omega_t^n}\right)^2 = t\ddot{\f_t} -\frac{t^2}{n} (\Delta_t \dot{\f_t})^2=  t\ddot{\f_t} -\frac{1}{n} (t\ddot{\f_t})^2.$$
In particular $t\ddot{\f_t} \leq n$ at $(t_0, x_0)$. Hence $H\leq n$, or equivalently $\ddot{\f_t} \leq \frac{n}{t}$ everywhere. Thus $t \mapsto \varphi_t - n (t\log t-t)$ is concave.

\smallskip

The function $t \in \R^+ \mapsto {\phi}_t(x)=\f_t(x)-n(t\log t-t) \in \R$  is concave,
hence $\dot{\phi}_t\leq \frac{\phi_t-\phi_0}{t}$. It follows that $\dot{\f}_t\leq \frac{\f_t-\f_0}{t}+n$.
The function  $t \in \R^+ \mapsto \tilde{\phi}_t(x)=\f_t(x)-n(t\log t-t) +C t \in \R$ is concave 
for   $x \in X \setminus (\f_0=-\infty)$,  and bounded below by $\f_0(x)$.
It follows 
that $\frac{d \tilde{\phi}_t}{dt} \geq 0$ 
for all times; the lower bound on $\dot{\varphi}_t$ follows.
\end{proof}

 \section{Uniform estimates} \label{sec:uniform}
 
  From now on we assume $\omega$ is K\"ahler and we fix a log smooth divisor $D=\sum_j D_j$, where the $D_j$'s are 
  {\it irreducible ample divisors with simple normal crossings}.

\subsection{Poincaré metrics} \label{subsect: Poincare}

We fix holomorphic \emph{ample} line bundles $L_1,\ldots,L_r$ equipped with Hermitian metrics $h_1,\ldots,h_r$ and, for any $j$, we denote by $\Theta_j$ the curvature of $h_j$.
For $s_i \in H^0(X,L_i)$, we denote by $|s_i|=|s_i|_{h_i}$ the norm of $s$ with respect to  $h_i$. 
Rescaling $h_i$, we can assume that $|s_i|\leq \frac{1}{1000}$ on $X$. 
We assume  $D_j=(s_j=0)$ is irreducible, and set
$$
D=\sum_{j=1}^r D_j
\; \; \text{ and } \; \;
|s|^2=\prod_{j=1}^r |s_j|^2.
$$

\begin{defi}
The divisor $D$ is {\it log smooth} if each divisor $D_j$ is smooth and the $D_j$'s have simple normal crossings,
 i.e. $D$ is locally isomorphic to $H_1 + H_2 + ...H_k$ where 
 $H_1,\ldots,H_k$ are the coordinate hyperplanes.  
\end{defi}

In the whole article we assume that $D$ is log smooth. Moreover, since each $L_j=\mathcal{O}(D_j)$ is ample, we can choose $h_j$ such that $ \Theta_j$ is a K\"ahler form. In particular $ \Theta:=\sum_{j=1}^r \Theta_j$ is a K\"ahler form as well. 
Fix $\tilde{\omega}$ yet another K\"ahler form.

\begin{defi}
 We denote by $\beta_P$ the following Poincaré metric along the divisor $D$,
 $$
 \beta_P:=\tilde{\omega} + \sum_{i=1}^r  dd^c [-2 \log(-\log |s_i|^2)].
 $$
\end{defi}

We call $\rho_P:=-2\sum_{j=1}^r \log(-\log |s_i|^2)$ a Poincar\'e potential.
 Our choice of normalization $|s_i|_{h_i} \leq \frac{1}{1000}$ ensures that 
 $\beta_P$ dominates (a multiple of) $\tilde{\omega}$. In local coordinates
 for which  $D_i=(z_i=0)$, the form $\beta_P$ is quasi-isometric to
 $$
 \beta_P'=\sum_{j=1}^r \frac{id z_j \wedge d\overline{z}_j}{|z_j|^2 (-\log |z_j|^2)^2}
 +\sum_{j=r+1}^n id z_j \wedge d\overline{z}_j.
 $$

We refer the reader to \cite{Kob84} for basic properties of the Poincaré metric $\beta_P$. 
 We simply mention here that the Riemannian curvature of $\beta_P$ is bounded in $X \setminus D$.

 \begin{defi}
  We say that $\beta$ is a K\"ahler form of Poincaré type
 in $\{ {\theta}\} \in H^{1,1}(X,\R)$ if 
 \begin{itemize}
 \item $\beta$ is a K\"ahler form in $X \setminus D$ and there exists $c>1$ such that $c^{-1} \beta_P \leq \beta \leq c \beta_P$;
 \item $\beta=\theta+dd^c u$, where $u$ is smooth in $X\setminus D$ and $|u-\rho_P|=O(1)$;
 \item $\left| \nabla_{\beta_P}^j u \right|$ is bounded in $X \setminus D$ for all $j \geq 1$.
 \end{itemize}
 \end{defi}
 
 A basic example of K\"ahler form of Poincaré type is provided by the unique solution of 
 $$
 (\tilde{\omega}+dd^c \rho)^n=\frac{e^{\rho}}{|s|^2} dV_X,
 $$
 that belongs to $\mathcal{E}(X,\tilde{\omega})$.

\subsection{Evolution of divisorial singularities}

We fix multiplicities $m_j\geq 0$ (not all zero) such that $\sum_j m_j \Theta_{h_j} \leq \omega$, and consider 
the initial datum
 $$
\f_0= \sum_{j=1}^r m_j \log |s_j|^2 \in PSH(X,\omega).
$$
We let $\f_t$ denote the solution of the maximal K\"ahler-Ricci flow starting from $\f_0$. 
In this section we estimate $\f_t$ up to bounded errors on the smoothing time interval $(0,\la_0(\f_0)]$.

\smallskip

The following result provides an asymptotic of $\f_t$ along  $(s_j=0)$.

\begin{lem} \label{lem:log}
Assume $\f_0= \sum_{i=1}^r m_i \log |s_i|^2+\p_0 \in PSH(X,\omega)$, where 
$\p_0$ a quasi-psh function.
There exists a constant $C>0$ such that  
	\[
	\varphi_t \leq  \sum_{i=1}^r \max(m_i- t,0) \log |s_i|^2 +Ct.
	\]
\end{lem}

Lemma \ref{lem:log} provides the right logarithmic behavior of 
the solution $\f_t$ when $D$ is
log smooth, and even when $D$ is merely 
{\it log canonical} (i.e. when  $c(\sum \log |s_j|^2)=1$).
This is no longer the case when $D$ is more singular, and one then needs to  use a log resolution
of the singularities to establish a refined first order asymptotic
(see \cite{DGL26}).

\begin{proof}
The proof is identical for one or several components (taking convex combination of each involved quantity), 
so we only treat the case $r=1$ and $m_1=1$ for simplicity.

We can assume without loss of generality that $\p_0=0$.
We approximate the initial potential $\varphi_0$ by $\varphi_{0,\varepsilon}= \log (|s|^2+\varepsilon)$ and we let $\varphi_{t,\varepsilon}$ denote the unique smooth solution to the Monge-Amp\`ere flow with initial datum 	$\varphi_{0,\varepsilon}$. 
Consider
$$
\p_{t,\e}=(1-t) \varphi_{0,\varepsilon} +C t
$$
and observe that 
$$
dd^c \varphi_{0,\varepsilon} 
 = \frac{-|s|^2}{(|s|^2+\varepsilon)} \Theta_1
 + \frac{\varepsilon }{(|s|^2+\varepsilon)^2}\frac{ i}{\pi} \partial s \wedge \overline{\partial s}. 
$$
Thus
$$
\omega+dd^c \p_{t,\e}=\left[ 1-\frac{(1-t) |s|^2}{|s|^2+\varepsilon} \right] \omega
+\frac{(1-t)|s|^2}{(|s|^2+\varepsilon)} (\omega-\Theta_h)
+ \frac{\varepsilon (1-t)}{(|s|^2+\varepsilon)^2} \frac{ i}{\pi}  \partial s \wedge \overline{\partial s} \geq 0,
$$
and 
$$
\omega+dd^c \p_{t,\e} \leq B \omega+ \frac{\varepsilon (1-t)}{(|s|^2+\varepsilon)^2} \frac{ i}{\pi}  \partial s \wedge \overline{\partial s}
\leq B \omega+ \frac{1}{|s|^2+\varepsilon} \frac{ i}{\pi}  \partial s \wedge \overline{\partial s}
$$
for some $B>0$.
Since $i \partial s \wedge \overline{\partial s}$ has rank $1$ and
 $\partial_t \p=-\log (|s|^2+\varepsilon) +C$, we infer
$$
(\omega+dd^c \p_{t,\e})^n \leq B^n \omega^n+ \frac{nB^{n-1}}{|s|^2+\varepsilon} \frac{ i}{\pi}  \partial s \wedge \overline{\partial s} \wedge \omega^{n-1}
\leq  \frac{e^C}{|s|^2+\varepsilon}  \omega^{n}=e^{\partial_t \p} \omega^n,
$$
if we choose $C>0$ large enough.

Since $\p_{0,\e}=\f_{0,\e}$, the classical maximum principle ensures that $\f_{t,\e} \leq \p_{t,\e}$.
The conclusion follows by letting $\e \rightarrow 0$, as $\f_{t,\e}$ then decreases  to $\f_t$.
\end{proof}

\subsection{Uniform estimates for small times}

\subsubsection{Purely divisorial initial datum}

When $T_0=\sum_{j=1}^r m_j[D_j]$ is purely divisorial
(i.e. $\omega=\sum_{j=1}^r m_j \Theta_{h_j}$) and $dV_X=\omega^n$,
we obtain an explicit decomposition of the metrics $\omega_t$.

\begin{prop} \label{pro:decompo}
If   $\omega=\sum_{j=1}^r m_j \Theta_{h_j}$ and $dV_X=\omega^n$, then for all $(x,t) \in X \times (0,\min_j m_j)$,
$$
\f_t=\sum_{j=1}^r \max(m_j-t,0) \log|s_j|^2 +t \rho +n(t\log t-t), 
$$
where $\rho \in {\mathcal E}(X,\Theta)$ is the unique solution of $(\Theta+dd^c \rho)^n =e^{\rho} |s|^{-2} \omega^n$,
with $|s|^2=\Pi_{j=1}^r |s_j|^{2}_{h_i}$.
\end{prop}

We thus have a relatively explicit solution to the flow when  $\omega=\sum_{i=1}^r m_i \Theta_{h_i}$. 

\begin{proof}
Set $\p_t=\sum_{i=1}^r \max(m_i-t,0) \log|s_i|^2+t \rho+n(t\log t-t)$. Observe that 
$\p_0=\f_0$ and $\omega+dd^c \p_t=t (\Theta+dd^c \rho)$ in $X \setminus D$,
hence
$$
(\omega+dd^c \p_t)^n=t^n (\Theta+dd^c \rho)^n=t^n \frac{e^{\rho}}{|s|^2} \omega^n=e^{\partial_t \p_t} \omega^n.
$$
In particular $\p_t \leq \f_t$ as $\p_t$ is a weak subsolution.

Now $\p_t$ is also a supersolution of the flow with $\p_0=\f_0$,
and all assumptions from Theorem \ref{thm:pcp} are satisfied, hence we conclude that $\f_t \leq \p_t$.
\end{proof}

If $m_1=\cdots=m_r=t_1$ we thus obtain an explicit description of the flow until the smoothing
time $\la(\f_0)=m$. If $t_1=\min_i m_i < \la(\f_0)=\max_i m_i$, understanding
the uniform estimate after time $t=m$ requires a finer analysis that we develop hereafter.

\begin{exa} \label{exa:notunique}
The same computation allows one to provide several examples of weak solutions to the
flow that have wilder asymptotic behavior near $D$. Consider indeed
$$
\p_{\alpha}=(1-\alpha t) \log |s|^2+ \alpha t \rho_{\alpha}+nt \log \alpha+n (t \log t-t),
$$
where $r=m_1=1$,
$0<\alpha<1$, $\omega= \Theta_{1}$,
and $\rho_{\alpha} \in {\mathcal E}(X,\omega)$ satisfies
$$
(\omega+dd^c \rho_{\alpha})^n=e^{\alpha \rho_{\alpha}} |s|^{-2\alpha} \omega^n.
$$
The metrics $\omega_{\alpha}=\omega+dd^c \p_{\alpha}$ are K\"ahler   in $X \setminus D$ with 
 logarithmic and conic singularities along $D$. They solve the  equation in 
$X \setminus D$, but   are   more singular than the maximal solution.
\end{exa}

\subsubsection{General case}

We now provide  the order zero asymptotic of the flow on $X \times (0,\min_i m_i)$ starting from $\varphi_0=\sum_j m_j \log|s_j|^2 $, without assuming  $\omega=\sum_{i=1}^r m_i \Theta_{h_i}$, nor $dV_X=\omega^n$.

\begin{prop} \label{prop:0uniform1}
Let $\rho \in PSH(X,\Theta)$ be the unique
solution of $(\Theta+dd^c \rho)^n =e^{\rho} |s|^{-2} dV_X$.
 There exists $C_0>0$ such that  
for all $(x,t) \in X \times (0,\min_i m_i)$,
$$
n(t\log t-t) \leq \f_t(x) - u_t(x) -t\rho(x) \leq   C_0 t,
$$
where $u_t=\sum_{i=1}^r \max(m_i-t,0) \log|s_i|^2$. 
\end{prop}

\begin{proof}
 Set $\tilde{\psi}_t=  \sum_j (m_j-t) \log|s_j|^2+t \rho+n(t\log t-t) $. Then, outside $D$
 $$\omega+ dd^c \tilde{\psi}_t = \omega+ \sum_j (m_j-t) dd^c \log |s_j|^2 +t dd^c \rho= \gamma +t(\Theta+dd^c \rho) \geq 0 $$ and
$$
		(\omega+dd^c \tilde{\psi}_t)^n \geq    t^n \Theta_\rho^n
		\geq t^n e^{\rho} |s|^{-2} dV_X = e^{\partial_t \tilde{\psi}_t} dV_X.
$$
	Consequently, the function $\tilde{\psi}_t$ is a subsolution to the Monge-Amp\`ere flow. By definition, this subsolution lies below $\varphi_t$.  
	
\smallskip

For the upper bound,   observe that  $\omega^j \wedge \Theta_{\rho}^{n-j} \leq C \Theta_{\rho}^n$ for any $j=1, \cdots, n$
since $\Theta_{\rho} \geq c \omega$. 
Thus for $\psi_t:= \sum_j(m_j-t) \log |s_j|^2+t\rho+C_0t $ we obtain (still outside $D$)
$$	
(\omega+dd^c {\psi}_t)^n =  \left ( \gamma+   t\Theta_\rho \right) ^n 
\leq \left ( A\omega+   t_1\Theta_\rho \right)^n \leq C_1 \Theta_\rho^n =e^{\partial_t \psi_t} dV_X. 
$$
It therefore follows  from Theorem \ref{thm:pcp}  that $\psi_t$ is a super-solution with $\f_t \leq \psi_t$.
\end{proof}

\subsection{Uniform estimates up to the smoothing time}

We assume here that the multiplicities $m_i$ are not equal and decompose 
the interval $(0,\la(\f_0)]$ accordingly,
$$
0<t_1=\min_i m_i <t_2 < \cdots <t_{\tilde{r}}=\max_i m_i=\la(\f_0).
$$
We let
$\tilde{D_1}=D_1 + \cdots + D_{i_1}, \ldots, \tilde{D}_{\tilde{r}}=D_{1+i_{{r}-1}} + \cdots + D_{r}$
denote the divisors that have the same multiplicity $\tilde{m}_j$,
let $\tilde{s_j}$ denote the product of the corresponding $s_i$'s so that
$\tilde{D_j}=(\tilde{s_j}=0)$.
With these notations one obtains
$t_1=\tilde{m}_1< \tilde{m}_2 < \cdots <\tilde{m}_{\tilde r}=\la_0(\f_0)$ and
$$
\f_0=\sum_{j=1}^{\tilde{r}} \tilde{m}_j \log |\tilde{s_j}|^2.
$$

\begin{prop} \label{prop:0uniform2}
Fix $1 \leq j < {\tilde r}-1$.
There exists $C_j>0$ such that for all $t_j <t \leq t_{j+1}$,
$$
2 \tilde{m}_j \log (t-t_j) -C_j t \leq 
\f_t -\sum_{\ell=j+1}^{\tilde{r}} (\tilde{m}_{\ell}-t) \log |\tilde{s_\ell}|^2 -t \hat{\rho}_j 
\leq C_j t,
$$
where $\hat{\rho}_j$ denotes the Poincaré potential along the log smooth divisor 
$\hat{D}_j=\tilde{D}_{j+1} + \cdots + \tilde{D}_r$.

\end{prop}

\begin{proof}
The proof goes by (finite) induction on $j$. To simplify the notations we only 
treat the case $j=1$. It follows from Proposition \ref{prop:0uniform1} that
\begin{eqnarray} \label{eq:c0long}
\f_{t_1}=\sum_{\ell=2}^{\tilde{r}} (\tilde{m}_{\ell}-\tilde{m}_{1}) \log |\tilde{s_\ell}|^2 +\tilde{m}_{1} {\rho}_D+O(1).
\end{eqnarray}
By the semigroup property, $\varphi_{t+t_1}$ is the flow starting from $\varphi_{t_1}$. Slightly abusing notations, 
we denote this flow by $\varphi_t$, with $t\in (0, t_2-t_1)$.
We decompose 
\[
{\rho}_D=-2 \sum_{m_j \leq t_1} \log (-\log |s_j|^2)+\hat{\rho_1},
\]
where
$\hat{\rho_1}$ is a Poincaré potential for the log smooth divisor 
$\hat{D}_1=\tilde{D}_{2} + \cdots + \tilde{D}_r$ 
whose components do not disappear at time $t_1$.
Consider
$$
\p_t=\sum_{\ell=2}^{\tilde{r}} (\tilde{m}_{\ell}-\tilde{m}_{1}-t) \log |\tilde{s_\ell}|^2 +(t+\tilde{m}_{1}) \hat{\rho}_1+C_1(t+1),  \qquad  t\in(0, t_2-t_1).
$$
It follows from \eqref{eq:c0long} that $\p_{0} \geq \f_{t_1}$
and $(\omega+dd^c \p_t)^n \leq e^{\partial_t \p} dV_X$ in $X \setminus \hat{D}_1$ if $C_1$ is large enough.
Lemma \ref{lem:log} ensures that for $t\in (0, t_2-t_1)$, 
\[
\varphi_t\leq \sum_{\ell=2}^{\tilde{r}} \max(\tilde{m}_\ell- \tilde{m}_1- t,0) \log |s_\ell|^2 +Ct.
\]
It follows therefore from Theorem \ref{thm:pcp} that $\f_t \leq \p_t$.

\smallskip

We now claim that 
$\f_t  \geq \p_t-C_2 (t+1)+2\tilde{m_1} \log t$
for $t \in (0,t_2-t_1)$, if $C_2>C_1$ is chosen large enough. Indeed consider 
$$
u_t=\p_t+|\tilde{s_1}|-C_2(1+t)+2\tilde{m_1} \log t.
$$
We are going to show that for each fixed $t \in (0,t_2-t_1)$, one has
\begin{eqnarray} \label{eq:MAelliptic}
e^{-{\f_t}/{t}}(\omega+dd^c \f_t)^n \leq e^{-{u_t}/{t}}(\omega+dd^c u_t)^n.
\end{eqnarray}
Accepting this, a direct computation shows that $\psi_t + \tilde{m}_1\rho_D$ is a subsolution to the flow starting from $\varphi_{t_1}$. Thus both $u_t$ and $\varphi_t$ belong to $ \mathcal{E}(X, \omega, P[\varphi_t])$ in the terminology of \cite{DDL2}. 
The relative comparison principle \cite[Corollary 3.6]{DDL2}  ensures  $u_t \leq \f_t$, 
providing the desired lower bound on $\f_t$, since the function $|s_1|$ is uniformly bounded.

\smallskip

We now establish \eqref{eq:MAelliptic}.
It follows from Proposition \ref{pro:concave}.(3) and the upper bound from previous step that 
$$
e^{-{\f_t}/{t}}(\omega+dd^c \f_t)^n =e^{-{\f_t}/{t}+\dot{\varphi}_t}  dV_X \leq e^{n-{\f_{t_1}}/{t}} dV_X.
$$
On the other hand, observing that 
$\partial \overline{\partial} |\tilde{s_1}|= \frac{1}{4} \frac{\partial \tilde{s_1}\wedge \overline {\partial \tilde{s_1}}}{|\tilde{s_1}|}$, we infer that 
there is a uniform constant $c>0$ such that 
$$
(\omega+dd^c u_t)^n \geq \frac{c}{|\tilde{s}_1|} \frac{dV_X}{|\hat{s}_1|^2 (-\log |\hat{s}_1|^2)^2},
$$
where   $\hat{s}_1=\Pi_{\ell=2}^{\tilde{r}} \tilde{s}_{\ell}$.
Thus 
$$
 \frac{e^{-{u_t}/{t}}(\omega+dd^c u_t)^n}{e^{-{\varphi_t}/{t}}(\omega+dd^c \varphi_t)^n} \geq 
 c\, e^{-n+\frac{\f_{t_1}-u_t}{t}}  \frac{1}{|\tilde{s}_1||\hat{s}_1|^2 (-\log |\hat{s}_1|^2)^2 }
$$
By definition
$$
\f_{t_1}-u_t= t \sum_{\ell=2}^{\tilde{r}}  \log |\tilde{s_\ell}|^2 -t\hat{\rho_1} +\tilde{m_1} (-2\log (-\log |\tilde{s_1}|^2) -|\tilde{s_1}|+(C_2-C_1)(1+t)-2\tilde{m_1} \log t,
$$
while $\prod_\ell |\tilde{s}_\ell|^2e^{-\hat{\rho}_1}= |\hat{s_1}|^2 (-\log |\hat{s_1}|^2)^2$. 
Using the fact that $|\tilde{s}_1|\leq C'$, we obtain  
\begin{eqnarray*}
 e^{\frac{\f_{t_1}-u_t}{t}} &\geq &  |\hat{s_1}|^2 (-\log |\hat{s_1}|^2)^2 (-\log|\tilde{s_1}|^2)^{-2\tilde{m}_1/t} e^{\frac{-|\tilde{s}_1|-2\tilde{m_1}\log t+(C_2-C_1)(1+t)}{t}}\\
 &\geq & |\hat{s_1}|^2 (-\log |\hat{s_1}|^2)^2 (-\log|\tilde{s_1}|^2)^{-2\tilde{m}_1/t} e^{C''+(C_3-2\tilde{m_1}\log t)/{t}},
  \end{eqnarray*}
  where $C''=C_2-C_1-C'$ and $C_3=C_2-C_1$.
 Combining the previous inequalities yields
 \begin{eqnarray*}
 \frac{e^{-{u_t}/{t}}(\omega+dd^c u_t)^n}{e^{-{\varphi_t}/{t}}(\omega+dd^c \varphi_t)^n} &\geq &  c\, e^{-n +C''+\frac{C_3-2\tilde{m}_1 \log t}{t}} \frac{1}{|\tilde{s}_1| (-\log|\tilde{s}_1|^2)^{2\tilde{m}_1/t}}\\
 &=&   c\, e^{-n +C''+\frac{C_3-2\tilde{m}_1 \log t}{t}} e^{-\frac{1}{2}\log |\tilde{s}_1|^2 - \frac{2\tilde{m}_1}{t}\log(-\log|\tilde{s}_1|^2)}.
 \end{eqnarray*}
Since $x=-\log |\tilde{s}_1|^2 \geq -C_3$ is   bounded below, the conclusion follows by noting that 
the function $x \mapsto \frac{x}{2}-\frac{2\tilde{m_1}}{t} \log x$ attains its minimum at $4\tilde{m}_1/t$, 
hence it is bounded below
by $ \frac{2\tilde{m_1}}{t} [1-\log 4\tilde{m}_1 +\log t]$ on $[-C_3,+\infty)$. Thus
$$\frac{e^{-{u_t}/{t}}(\omega+dd^c u_t)^n}{e^{-{\varphi_t}/{t}}(\omega+dd^c \varphi_t)^n} \geq e^{C'''} e^{\frac{2\tilde{m}_1}{t} (1-\log 4\tilde{m}_1+C_3)}$$
Choosing $C_2$   large enough so that $(1-\log 4\tilde{m}_1+C_3) \geq 0$ concludes the proof.
\end{proof}

\section{Higher order estimates}  \label{sec:logsmooth}

To achieve the higher order analysis we again need to  work in each time interval $(t_{\ell},t_{\ell+1})$. 
 Thanks to Proposition \ref{prop:0uniform2} we have the following decomposition in $(t_\ell, t_{\ell+1})$,
 $$
\f_t=\sum_{m_j>t_{\ell}} (m_j-t) \log |s_j|^2+t \rho_{\ell}+v_t^\ell,
$$
where $ 2 \tilde{m}_\ell \log (t-t_\ell) -C_\ell  \leq v_{t}^{\ell} \leq C_\ell t$.

 \subsection{Time derivative}
 
\begin{prop} \label{pro:STbelow}
There exists a uniform constant $C_\ell>0$ such that 
 $$
 n \log (t-t_{\ell}) -C_\ell \leq \frac{\partial v_t^\ell}{\partial t}, 
$$
for all $(x,t) \in X \times (t_\ell, t_{\ell +1}]$.
\end{prop}

\begin{remark}
 One can also prove that $\frac{\partial v_t}{\partial t} \leq C$, when $(x,t)$ belongs to $X \times (0,t_1]$. 
 This  follows from Proposition \ref{pro:concave} and Proposition \ref{prop:0uniform1}. Indeed,
	\[
	-\sum_j \log|s_j|^2+\rho+\dot{v}_t=\dot{\varphi}_t \leq \frac{\varphi_t-\varphi_0}{t} +n \leq \rho-\sum_j \log|s_j|^2 +n+A_2,
	\]
	where $A_2$ is the positive constant appearing in Proposition \ref{prop:0uniform1}. 
\end{remark}

\begin{proof}
The proof is inspired by the parabolic pluripotential approach developed in \cite{GLZ20}.

We first treat the case $\ell=0$, i.e. when $t \in (0, \min_i m_i]$.
Set $\gamma = \omega-\sum_jm_j\Theta_j \geq 0$,
fix $\varepsilon,c>0 $ and consider 
	\[
	w_t = (1-\varepsilon c^{-1}) \varphi_t + \varepsilon c^{-1}  \psi_t,
	\]
	where 
	\[
	\psi_t = \varphi_0 + (c+1)t \left(\rho-\sum_{j=1}^r \log|s_j|^2\right). 
	\]
By the fact that $\omega+dd^c \varphi_0=\gamma+[D]$ and a direct computation we obtain that outside $D$
	\[
	(\omega+dd^c \psi_t) = \gamma + t (c+1)\Theta_\rho \geq t (c+1) \Theta_\rho
	\]
	hence

	\[
	(c+1)^{-n}(\omega+dd^c \psi_t)^n \geq e^{n\log t+ (c+1)^{-1}\dot{\psi}_t} dV_X. 
	\]
	
Recall that, if $\theta_1, \theta_2$ are two positive forms with $\theta_1^n\geq e^{f_1} dV_X$ and $\theta_2^n\geq e^{f_2}dV_X$, 
it follows from the the mixed Monge-Amp\`ere inequalities and the convexity of 
the exponential that for $ \lambda_1, \lambda_2 \in [0,1]$ such that $\lambda_1+\lambda_2=1$,
	\begin{eqnarray*}
	(\lambda_1 \theta_1+ \lambda_2\theta_2)^n & \geq &  
	\sum_{k=0}^n \binom{n}{k}  \lambda_1^k \lambda_2^{n-k}  e^{\frac{k}{n} f_1+\frac{n-k}{n} f_2}  dV_X \\
	&=& (\lambda_1 e^{f_1/n} + \lambda_2 e^{f_2/n})^n dV_X\\
	&\geq& e^{\lambda_1 f_1+\lambda_2 f_2}  dV_X.
	\end{eqnarray*}
	Using this inequality and the identity $\frac{1+c^{-1}}{1+c}=c^{-1}$, we obtain
	\begin{flalign*}
	(\omega + dd^c w_t)^n & = (1+\varepsilon)^n\left ( \frac{1-c^{-1}\varepsilon}{1+\varepsilon} \omega_{\varphi_t} + \frac{c^{-1}\varepsilon}{1+\varepsilon} \omega_{\psi_t} \right )^n 	\\
    &= (1+\varepsilon)^n \left ( \frac{1-c^{-1}\varepsilon}{1+\varepsilon} \omega_{\varphi_t} + \frac{(1+c^{-1})\varepsilon}{1+\varepsilon} \frac{\omega_{\psi_t}}{c+1} \right )^n \\
	&\geq \exp \left (\frac{1-c^{-1}\varepsilon}{1+\varepsilon} \dot{\varphi}_t + \frac{(1+c^{-1})\varepsilon}{1+\varepsilon} (c+1)^{-1} \dot{\psi}_t + \frac{(1+c^{-1})\varepsilon}{1+\varepsilon} n\log t \right ) dV_X\\
    &=\exp \left (\frac{1-c^{-1}\varepsilon}{1+\varepsilon} \dot{\varphi}_t + \frac{c^{-1}\varepsilon}{1+\varepsilon}  \dot{\psi}_t + \frac{(1+c^{-1})\varepsilon}{1+\varepsilon} n\log t \right ) dV_X\\
	&\geq \exp \left ( \frac{\dot{w}_t}{1+\varepsilon} +\frac{(1+c^{-1})\varepsilon}{1+\varepsilon} n\log t \right ) dV_X.
	\end{flalign*}
	On the other hand, the function $\varphi_{t}^{\varepsilon}:= \varphi_{(1+\varepsilon)t}$ satisfies 
	$\varphi_0^{\varepsilon} = w_0= \varphi_0$ and 
	\[
	(\omega+dd^c \varphi_t^{\varepsilon})^n = \exp \left( \frac{\dot{\varphi}_t^{\varepsilon}}{1+\varepsilon} \right ) dV_X. 
	\]
	Thus $w_t+ (1+c^{-1}) n \varepsilon t(\log t-1)$ is a subsolution of this rescaled
	flow, hence
	\[
	\varphi_{(1+\varepsilon)t} \geq w_t + (1+c^{-1}) n \varepsilon t(\log t-1).
	\]
	Using  the definition of $w_t$ and the upper bound for $v_t$ provided by Proposition \ref{prop:0uniform1}, 
	we infer 
	\begin{eqnarray*}
	\lim_{\varepsilon\to 0} \frac{\varphi_{(1+\varepsilon)t}-\varphi_t}{\varepsilon t} & \geq & \frac{w_t-\varphi_t}{\varepsilon t} - C_ + (1+c^{-1})n \log t \\&\geq & \frac{\psi_t-\varphi_t}{ct} - C + (1+c^{-1})n \log t \\
    &\geq & \rho - \sum_j \log |s_j|^2 - \frac{v_t}{ct} - C + (1+c^{-1})n \log t \\
	&\geq & \left(\rho-\sum_j \log |s_j|^2\right)-C_0c^{-1} -C + (1+c^{-1})n \log t.
	\end{eqnarray*}
The lower bound for $\dot{v}_t$ follows since $\dot{\varphi}_t=-\sum_j \log |s_j|^2+\rho+\dot{v}_t$.

\smallskip

    The estimates for other time intervals follow from similar arguments, with a more singular initial datum.
    We assume that $\varphi_0=\sum_{j=1}^r m_j\log |s_j|^2 + h$, where $h\in \mathcal{E}(X,\omega/2)$, while $\sum_{j=1}^r \log |s_j|^2 \in PSH(X,\omega/2)$.  We first assume $h$ is smooth and establish estimates independent of $h$. By the previous step, we have 
    \[
    \dot{\varphi}_t \geq -\sum_{j=1}^r \log |s_j|^2+ \rho + n\log t -C,
    \]
    with $C>0$ that a priori depends on $h$.
    We will show that $C$ can be chosen independently of $h$. 
    We fix $\varepsilon>0$ and consider 
\[
H(t,x):= \dot{\varphi}_t + (1-\varepsilon)\sum_{j=1}^r\log |s_j|^2  + 2 \sum_{j=1}^r\log (-\log |s_j|^2)  - n(1+\varepsilon) \log t. 
\] 
The previous step ensures that $H$ attains its minimum at some point $(t_0,x_0)$ such that
$t_0>0$ and $s_j(x_0)\neq 0$. At this point we obtain
	\begin{flalign*}
			0& \geq (\partial_t -\Delta_t) H  \\
			& \geq - \frac{n(1+\varepsilon)}{t}  + \sum_{j=1}^r{\rm Tr}_t\left ((1-\varepsilon)\Theta_j- \frac{2 \Theta_j}{-\log |s_j|^2}  +  2|s_j|^{-4} (\log |s_j|^2)^{-2}d |s_j|^2 \wedge d^c |s_j|^2\right ). 
	\end{flalign*}
Here, $\Theta_j:= -dd^c \log |s_j|^2\geq c_0 \omega$ outside $\{s_j=0\}$. 
Since $|s_j|<1/1000$ we infer
\begin{flalign*}
0&\geq (\partial_t -\Delta_t) H  \\
			& \geq - \frac{n(1+\varepsilon)}{t}  + \sum_{j=1}^r{\rm Tr}_t\left (c_0\omega+   2|s_j|^{-4} (\log |s_j|^2)^{-2}d |s_j|^2 \wedge d^c |s_j|^2\right )\\
            & \geq - \frac{n(1+\varepsilon)}{t} + \frac{nc_1e^{-\dot{\varphi}_t/n}}{\prod_{j=1}^r |s_j|^2 (\log |s_j|^2)^2},
\end{flalign*}
using the inequality ${\rm Tr}_{\beta}(\alpha) \geq n  (\alpha^n/\beta^n)^{1/n}$. It thus follows that $H(t_0,x_0)\geq -C_2$, with $C_2$ independent of $\varepsilon$ and $h$. The result follows. 
\end{proof}

\subsection{Poincaré type metrics along the flow}

   Recall that on each interval $(t_\ell, t_{\ell+1}]$,
   $$
\omega_t=\sum_{m_i\geq t_\ell}^r (m_i-t) [D_i]+t \beta_{t, \ell},
$$
 where $\beta_{t, \ell}$ is a  K\"ahler form in  $X \setminus D$, and
  a positive closed current on $X$ with no mass along  $D$.
 We now show   that $\beta_{t, \ell}$ is a Poincaré type metric along $\hat{D}_\ell$.

  \subsubsection{${\mathcal C}^2$-estimate}
   
   We first show that $\beta_{t, \ell}$ is quasi-isometric to $\omega_{P, \ell}$, 
   the Poincar\'e metric associated to the divisor $\hat{D}_\ell$.
 
 \begin{theorem} \label{thm:c2hyp}
	 There exists $C_t>1$ such that on each interval of time $(t_\ell, t_{\ell+1})$ we have   
	 $$
	 C_t^{-1} \omega_{P,\ell} \leq \beta_{t,\ell} \leq C_t \, \omega_{P, \ell}.
	 $$
\end{theorem}

The key step 
is the following upper-bound:

\begin{lem} \label{lem:c2baby}
There exists a uniform constant $C_{\ell}>0$ such that for all $t\in (t_\ell, t_{\ell+1})$,
\[
\log \tr_{\omega_{P,\ell}} (\omega_{t}) \leq C_{\ell} \frac{|\log (t-t_{\ell})|}{(t-t_{\ell})^2}.
\]
\end{lem}

\begin{proof} 
To simplify notations we write $\beta_t,\omega_P$ instead of $\beta_{t, \ell},\omega_{P,\ell}$,
and work on the time interval $(0,t_{\ell+1}-t_{\ell})$ by replacing $t$ by $t-t_{\ell}$.
By Proposition \ref{prop:0uniform2} we can decompose
$$
\f_t =\sum_{k=\ell+1}^{\tilde{r}} (\tilde{m}_{k}-t) \log |\tilde{s_k}|^2 +\psi_t,
$$
where  $\p_t=t \hat{\rho}_\ell +v_t$, and 
$
-tA_2  +n(\log (t-t_\ell) -t) \leq v_t \leq A_1t.
$
Hence in $X \setminus D$ we have
$$
\omega_t^n\geq t^n \beta_t^n=e^{\partial_t \f_t} dV_X
=e^{\partial_t v_t} \frac{e^{\hat{\rho_\ell}}}{\prod_{\tilde{m}_k \geq t_\ell}|s_k|^2}dV_X,
$$
with $\beta_t=\sum_{k>\ell} \Theta_k+dd^c {\hat{\rho}_\ell}+dd^c (v_t/t)$.

 \smallskip
 
Let $\varphi_{t,j}$ be the solution of the flow starting from 
$$
\varphi_{t_\ell,j}=\sum_{k=\ell}^{\tilde{r}} (\tilde{m}_{k}-\tilde{m}_{1}) \log (|\tilde{s}_k|^2+1/j) -2\tilde{m}_{\ell}\sum_{ \tilde{m}_{k} \geq t_\ell} \log \left(-\log (|\tilde{s}_k|^2+1/j )\right) 
$$
and define $v_{t,j}$ by $\varphi_{t,j}=\sum_{k=\ell}^{\tilde{r}} (\tilde{m}_{k}-t) \log |\tilde{s}_k|^2 +t \hat{\rho_\ell}+v_{t,j}$. 
Let us stress that $v_{t,j}$ is not smooth but it is uniformly bounded from above in $j$ thanks to Proposition \ref{prop:0uniform2}. 
We also stress that $v_{t,j}$ is not (a priori) uniformly bounded from below.

Fix $\e>0$ and set 
$$
\alpha:=t \log u_j-A v_{t+\e,j}+ \eta \log |s|^2,
$$
 where $\eta>0$ is a small constant, 
$u_j:=\tr_{\omega_P} (\omega_{t+\e,j}) $ and $A>0$ will be specified later. 
We are going to establish estimates that are uniform in $j,\eta$, and let eventually
$j \rightarrow +\infty$ and $\eta \rightarrow 0$.

The function $\log u_j$ is bounded from above since $\omega_{t+\e,j}$ is smooth and $\omega_P$ is a K\"ahler current.
It therefore follows from the uniform ${\mathcal C}^0$-estimate that
$\alpha$ reaches its maximum on $[0,T] \times X \setminus D$.
If the maximum of $\alpha$ is reached at $t_0=0$, then
$$
t \log u_j (x) \leq -A v_{\varepsilon, j}(x_0) + Av_{t+\varepsilon, j}(x) -\eta \log |s|^2
\leq C_1 \left(1 -\log \varepsilon \right)  -\eta \log |s|^2. 
$$
Taking $\varepsilon=t$ and sending $\eta$ to zero, and then $j$ to $+\infty$ we obtain
$$
\log \tr_{\omega_P} (\omega_{2t}) \leq \frac{C_2}{t}(-\log t).  
$$

We now assume that $\alpha$ reaches its maximum at  $(t_0,x_0)$ with
$t_0>0$ and  $x_0\in X\setminus D$. We remove the subscript $j$ to simplify notations. 
Set $\Delta=\Delta_{\omega_{t+\e}}$ and observe that
$$
\left(\frac{\partial}{\partial t}-\Delta \right)(\alpha)
=\log u + \frac{t}{u} \frac{\partial u}{\partial t}-A \frac{\partial v_{t+\e}}{\partial t} 
-t \Delta \log u+A \Delta (v_{t+\e}) + \eta \tr_{\omega_{t+\varepsilon}} (\Theta).
$$
Since $\omega_{t+\e}=(t+\e) \omega_P+dd^c v_{t+\e}$ in $X \setminus D$,
the last but one term is 
$$
A \Delta (v_{t+\e})=An -A(t+\e) \tr_{\omega_{t+\e}}(\omega_P).
$$ 

We let $-B$ denote a lower-bound on the holomorphic bisectional curvature of $\omega_P$ in $X \setminus D$.
The last but two terms are estimated thanks to Lemma \ref{lem:Siu},
$$
-t \Delta \log u \leq Bt \, \tr_{\omega_{t+\e}}(\omega_P)+
t \frac{\tr_{\omega_P}({\rm Ric} (\omega_{t+\e}))}{\tr_{\omega_P}(\omega_{t+\e})}.
$$

Since
$$
\frac{t}{u} \frac{\partial u}{\partial t}=\frac{t}{u} \Delta_{\omega_P} (\log \omega_{t+\varepsilon}^n/\omega^n)
=\frac{t}{u} \left\{ -\tr_{\omega_P}({\rm Ric} \, \omega_{t+\varepsilon}) +\tr_{\omega_P}({\rm Ric} \, \omega) \right\},
$$
we infer 
$$
-t \Delta_t \log u+\frac{t}{u} \frac{\partial u}{\partial t} \leq 
 (B+C_1) t \, \tr_{\omega_{t+\varepsilon}}(\omega_P),
$$
using that $\tr_{\omega_P}({\rm Ric} \,\omega) \leq C \tr_{\omega_P}(\omega)  \leq C'$ 
and  $u^{-1}= [\tr_{\omega_P}(\omega_{t+\e})]^{-1} \leq n^{-1} \tr_{\omega_{t+\e}}(\omega_P)  $.

\smallskip

Observe now that 
$$
u=\tr_{\omega_p}(\omega_{t+\e}) 
\leq n \left( \frac{\omega_{t+\e}^n}{\omega_P^n} \right) \cdot \left( \tr_{\omega_{t+\e}} (\omega_P) \right)^{n-1}
\leq C  e^{\partial_t v_{t+\e}} \left( \tr_{\omega_{t+\e}} (\omega_P) \right)^{n-1}.
$$
The inequality  $(n-1)\log x <x+C_n$ therefore yields
\begin{equation} \label{eq:trace2bis}
\log u \leq \frac{\partial v_{t+\e}}{\partial t} +C_2+ \tr_{\omega_{t+\e}}(\omega_P).
\end{equation}
 Altogether we obtain
\small{
\begin{flalign*}
\left(\frac{\partial}{\partial t}-\Delta\right)(\alpha) &
\leq (An+C_2)- (A-1) \frac{\partial v_{t+\e}}{\partial t} +\left[ (B+C_1) t +1-A(t+\e) \right] \tr_{\omega_{t+\varepsilon}}(\omega_P)+ \eta \tr_{\omega_{t+\varepsilon}} (\Theta)\\
& \leq C_3- (A-1) \frac{\partial v_{t+\e}}{\partial t} +\left[ (B+C_1) t +2-A(t+\e) \right] \tr_{\omega_{t+\varepsilon}}(\omega_P),
\end{flalign*}
}
using $\Theta \leq C_4 \omega_P$ and taking $\eta$ sufficiently small. 
We choose $A=A(\varepsilon)=(B+C_1)+3/\varepsilon>2$ so that $(B+C_1)t +2-A(t+\e) \leq -1$. 
Observe that $A \leq C/\varepsilon$ for some uniform $C>0$.

Let $(x_0, t_0)$ be a point at which $\alpha$ reaches its maximum,
with $t_0>0$ and $x_0 \in X \setminus D$ as explained above.
The maximum principle and Proposition \ref{pro:STbelow} ensure that 
$$ 
\frac{\partial v_{t_0+\e}}{\partial t}  + \tr_{\omega_{t_0+\varepsilon}}(\omega_P)(x_0)\leq C_3-(A-2)(n\log (t_0+\varepsilon)-C). 
$$  
It then follows from \eqref{eq:trace2bis} that
$$
\log u(x_0, t_0) \leq  C_3-(A-2)(n\log (t_0+\varepsilon)-C) +C_2
\leq  C(\varepsilon).
$$ 
Since $\log|s|^2 \leq 0$ and $v_{t+\varepsilon}$ is bounded from below by $ \log (t+\varepsilon) $ (uniformly in $j$),  
we obtain 
$$
\alpha (x_0, t_0)\leq t_0  C(\varepsilon) -A(\varepsilon)C_5 \log (t_0+\varepsilon).
$$
Thus, since $v_{t+\varepsilon}$ is bounded from above thanks to Proposition \ref{prop:0uniform2}.
$$
t\log u_j \leq  t_0 C(\varepsilon)  -A(\varepsilon)C_5 \log (\varepsilon)+A(\varepsilon)C_6 -\eta \log|s|^2.
$$
Taking $\varepsilon=t$ we obtain
$
t\log u_j \leq  C_7 (1-\log t)  -\tilde{C}_5 \frac{\log t}{t}+\frac{\tilde{C}_6}{t} -\eta \log|s|^2,
$
as $A\sim \frac{1}{t}$.
The desired inequality follows  by letting $\eta \to 0$ and $j \rightarrow +\infty$. 
\end{proof}

\begin{proof}[Proof of Theorem \ref{thm:main}]
Recall that $\omega_t=t \beta_t$ in $X \setminus D$.
It follows from Lemma \ref{lem:c2baby} applied with $\e=t$ that $\beta_t \leq C_t \omega_P$.
Conversely Proposition \ref{pro:STbelow} shows that $\omega_P^n \leq C_t'  \beta_t^n $.
The inequality
$
\tr_{\beta_t}(\omega_P)
\leq n \left( \frac{\omega_P^n}{\beta_t^n} \right) \cdot \left( \tr_{\omega_P} (\beta_t) \right)^{n-1}
$
allows one to conclude.
\end{proof}

 We have used the following classical  inequalities, see \cite{Siu87}.
 
 \begin{lem} \label{lem:Siu}
Let $\omega,\omega'$ be arbitrary K\"ahler forms.
Let $-B \in \R$ be a lower bound on the holomorphic bisectional curvature of $(X,\omega)$. Then
$$
\Delta_{\omega'} \log \tr_{\omega} (\omega') \geq -\frac{\tr_{\omega}(Ric(\omega'))}{\tr_{\omega}(\omega')}
-B \,  \tr_{\omega'}(\omega).
$$
\end{lem}

  \subsubsection{Higher order estimates}
  
  We can now conclude the proof of our main result.
  
  \begin{theorem} \label{thm:main}
  Let $(\omega_t)_{t>0}$ be the solution to the twisted K\"ahler-Ricci flow emanating from $T_0=\sum_{i=1}^r m_i[D_i]$.
Let $j(t)$ denote the greatest index such that $m_{j(t)}>t$.
If $D=\cup D_i$ is log-smooth, then $\omega_t$ is a K\"ahler form in $X \setminus D$ such that
   $$
   \omega_t=\sum_{m_i>t} (m_i-t)  [D_i] + 	\beta_{t},
   $$
   where $\beta_{t}$ is a  Poincar\'e type metric in $X \setminus S_t$,
   where $S_t=D_1 \cup \cdots D_{j(t)}$. In particular, $\omega_t$ has bounded curvature in $X\setminus S_t$. 
   
Moreover $\omega_t$ is a global K\"ahler form when $t > \la(\f_0)$.
  \end{theorem}

  \begin{proof}
  We again work on each time interval $(t_{\ell},t_{\ell+1}]$. 
  Set $\tilde{m}:=m_{\ell+1}-\sum_{j=1}^{\ell} m_j$. 
  Suppressing the index $\ell$ and shifting the interval to $[0,\tilde{m}]$, we use Proposition \ref{prop:0uniform2} to rewrite the flow on $[0,\tilde{m}]\times X\setminus \cup_{m_j\leq t_{\ell+1}}D_j$ as 
 \[
 \left( \omega + tdd^c \rho_P -\sum_j (m_j-t) \Theta_j+ dd^c v_t\right)^n = e^{\dot{v}_t+F}\omega_P^n,
 \]
with $v_0=\sum_{j}b_j(-2\log (-\log |s_j|^2))$ and $b_j\geq 0$, $\rho_P$  a Poincar\'e potential along $\cup_{m_j\leq t_{\ell+1}}D_j$, and $F$ a smooth function.
  Fixing a small constant $\varepsilon>0$, it follows again from Proposition \ref{prop:0uniform2} that $v_{t}$ is bounded when $t\in [\varepsilon,\tilde{m}]$. From (3) of Proposition \ref{pro:concave} we have
  $$\dot{\varphi}_t = -\sum_j \log |s_j|^2 +\rho_P +\dot{v}_t\leq \frac{t(\rho_P-\sum_j \log |s_j|^2)+v_t}{t}.$$ As $v_t\leq C t$, we infer that $\dot{v}_t$ is bounded from above for $t\in [\varepsilon,\tilde{m}]$.
  Proposition \ref{pro:STbelow} gives the bound from below. 
  Here, when $t=\tilde{m}$,  $\dot{v}_t$ stands for the left derivative.  Theorem \ref{thm:c2hyp}, and the complex parabolic Evans-Krylov theory together with Schauder’s estimates then ensure that $(t,x) \mapsto \f_t(x)$ is smooth in $X \setminus S_{0} \times (0,\tilde{m})$.

We now prove that $\beta_t:= \omega + tdd^c \rho_P -\sum_j (m_j-t) \Theta_j+ dd^c v_t$ is of Poincar\'e type.
Taking the pull back by the quasi-coordinates $\psi_{\eta}$, as defined by Kobayashi in \cite{Kob84} 
(see also \cite{Gue14,Auv17}), the above equation becomes
\[
\left( \psi_{\eta}^*\alpha_t + dd^c v_t\circ \psi_{\eta}\right )^n=e^{\dot{v}_t\circ \psi_{\eta}+ F\circ \psi_{\eta}} \beta^n
\]
in a polydisc in $\mathbb C^n$. Here $\alpha_t:=\omega-\sum_j(m_j-t)\Theta +tdd^c \rho_P$ and $\beta=\psi_{\eta}^* \omega_P$ is a smooth closed $(1,1)$-form independent of $\eta$, which is comparable to the euclidean form. Writing $\beta=dd^c g$, $\psi_{\eta}^*\alpha_t=dd^c h_{t}$, $u_{\eta,t}=v_t\circ \psi_{\eta}+h_t$, the above equation reduces to 
\[
(dd^c u_{\eta,t})^n =e^{\dot{u}_{\eta,t}+G} \beta^n,
\]
where $G=F\circ \psi_{\eta}-\dot{h}_t$ is smooth and uniformly 
bounded in any ${\mathcal C}^k$-norm independent of $\eta$. 

Fixing small $\varepsilon>0$, we know that $v_t, \dot{v}_t$ are bounded for $t\in [\varepsilon,\tilde{m}]$. From Theorem \ref{thm:c2hyp}, we infer that the Laplacian of $u_{\eta,t}$ is bounded when $t\in [\varepsilon,\tilde{m}]$. Thus from Schauder's estimate, see \cite[Theorem 4.1.4]{BG12}, 
we obtain ${\mathcal C}^k(U')$-estimates of $u$ for all $k$, where $U'\Subset U$.

Covering $X$ by finitely many $(U',U)$ we see that, for $t\in (\varepsilon, \tilde{m}]$, the pull back of $\beta_t$ in any quasi-coordinates system is bounded in 
${\mathcal C}^k$-norm for any $k$. It then follows that $\omega_t$ has bounded curvature in $X\setminus X_t$ (see \cite{TY86} for more details).  
\end{proof}


\begin{thebibliography}{99}

\bibitem[ALS24]{ALS24} O.
Alehyane, C.H.Lu, and M.Salouf. 
{\it Degenerate Complex Monge–Ampère Equations on Some Compact Hermitian Manifolds}. 
J. Geom Anal 34, 320 (2024). 

\bibitem[Auv17]{Auv17} H.Auvray,
{\em The space of Poincar\'e type K\"ahler metrics on the complement of a divisor. }
J. Reine Angew. Math. 722 (2017), 1-64.
 
 \bibitem[BM87]{BM87} S.Bando, T.Mabuchi,
 {\em Uniqueness of Einstein K\"ahler metrics modulo connected group actions,}
  in Algebraic geometry, Sendai, 1985, Adv. Stud. Pure Math. 10, Kinokuniya, 1987, 11-40.
 


\bibitem[BBJ21]{BBJ21} R.Berman, S.Boucksom, M.Jonsson, 
{\em A variational approach to the Yau-Tian-Donaldson conjecture.}
 J. Amer. Math. Soc. 34 (2021), no. 3, 605-652. 
 
 \bibitem[BG14]{BG14} R.Berman,  H.Guenancia,
{\em K\"ahler-Einstein metrics on stable varieties and log canonical pairs.}
G.A.F.A. 24 (6), 1683-1730 (2014).



\bibitem[BG12]{BG12} S.Boucksom, V.Guedj, {\it Regularizing properties of the K\"ahler-Ricci flow}, Lect. Notes Math. 2086, 189--237 (2013). 

 
\bibitem[BGL25]{BGL25} S.Boucksom, V.Guedj, C.H.Lu,  
{\em Volumes of Bott-Chern classes.}
 Peking Math J (2025). 
  

\bibitem[Dan25]{Dan25} Q.T.Dang, {\it Singularities of the Chern-Ricci flow}, Anal. PDE 19-3 (2026), 449-483.
 
 
  
  \bibitem[DDL18]{DDL2} T.Darvas, E.Di Nezza, and C.H.Lu,
   {\em Monotonicity of nonpluripolar products and complex {M}onge-{A}mp\`ere equations with prescribed singularity}, 
   {Analysis \& PDE}, {11}, {2018}, no. {8}, {2049-2087}.

 \bibitem[DDL21]{DDL21} T.Darvas, E.Di Nezza, C.H.Lu. 
 {\it The metric geometry of singularity types}. 
 Journal f\"ur die reine und angewandte Mathematik, vol. 2021, no. 771, 2021, 137--170. 
                     
 
\bibitem[Dem92]{Dem92} J.P.~Demailly: 
{\em Regularization of closed positive currents and intersection theory}. 
J. Algebraic Geom.  {\bf 1}  (1992),  no. 3, 361--409. 

 
\bibitem[DK01]{DK01} J.P.Demailly, J.Koll\'ar, 
{\em Semi-continuity of complex singularity exponents and K\"ahler-Einstein metrics on Fano orbifolds.} 
Ann. Sci. Ec. Norm. Sup. (4) { 34} (2001), no 4, 525-556.


\bibitem[DPS01]{DPS01} J.P.Demailly, T.Peternell, M.Schneider,
{\em Pseudo-effective line bundles on compact K\"ahler manifolds.}
 Internat. J. Math. 12 (2001), no. 6, 689-741. 
 


 
  \bibitem[DGG23]{DGG23} E.Di Nezza, V.Guedj, H.Guenancia, 
  {\em Families of singular K\"ahler-Einstein metrics.}
  J. Eur. Math. Soc. 25 (2023), n°7, 2697-2762.
 

 
  \bibitem[DNL17]{DNL17} E.Di Nezza, C.H. Lu,  
  {\em Uniqueness and short time regularity of the weak K\"ahler-Ricci flow.}
  Adv. Math. { 305} (2017), 953--993.
  
   \bibitem[DGL26]{DGL26} E.Di Nezza, V.Guedj, C.H.Lu,  
  {\em Geometric smoothing by the K\"ahler-Ricci flow}.
  Preprint (2026).
 

 
  \bibitem[DGZ16]{DGZ16} S.Dinew, V.Guedj, A.Zeriahi, 
  {\em Open problems in pluripotential theory.}
   Complex Var. Elliptic Equ. 61 (2016), no. 7, 902-930. 
   
   \bibitem[EGZ11]{EGZ11}  P.Eyssidieux, V.Guedj, A.Zeriahi, 
   {\em Viscosity solutions to Degenerate Complex Monge-Amp\`ere Equations.} 
  Comm. Pure Appl. Math.  { 64}  (2011),  no. 8, 1059--1094. 
 
 
 
   \bibitem[GLZ20]{GLZ20} V.Guedj, C.H.Lu, A.Zeriahi,  
  {\em Pluripotential K\"ahler-Ricci flows}. 
 Geom. Topol. 24 (2020), 1225-1296.
  

  \bibitem[GLZ21]{GLZ21} V.Guedj, C.H.Lu, A.Zeriahi,  
  {\em The pluripotential Cauchy-Dirichlet problem for complex Monge-Amp\`ere flows}. 
  Ann. Sci. \'Ec. Norm. Sup\'er. (4) 54 (2021), no. 4, 889-944.
  
  
 \bibitem[GZ07]{GZ07} V.Guedj, A.Zeriahi,   
{\em The weigthed Monge-Amp\`ere energy of quasiplurisubharmonic functions.}
J. Funct. Anal. 250 (2007), 442-482.

\bibitem[GZ17]{GZ17} V.Guedj, A.Zeriahi,   
{\em Regularizing properties of the twisted K\"ahler-Ricci flow.} 
Journal f\"ur die reine und ang. Math., { 729} (2017), 275--304. 


\bibitem[GZ]{GZbook} V.Guedj, A.Zeriahi,  
{\em Degenerate Complex Monge-Amp\`ere Equations}. 
EMS Tracts in Math. { 26} (2017).

\bibitem[Gue14]{Gue14} H. Guenancia, {\it K\"ahler-Einstein metrics with mixed Poincar\'e and cone singularities along a normal crossing divisor}, Annales de l'Institut Fourier, Tome 64 (2014) no. 3, pp. 1291--1330.

  
 \bibitem[Kob84]{Kob84}  R.Kobayashi, 
 {\em K\"ahler-Einstein metric on an open algebraic manifold.}
  Osaka J. Math. 21 (1984), no. 2, 399-418. 


 
   \bibitem[PSSW08]{PSSW08} D.H.Phong, J.Song, J.Sturm, B.Weinkove,
   {\em The Moser-Trudinger inequality on K\"hler-Einstein manifolds.}
    Amer. J. Math. 130 (2008), no. 4, 1067-1085.
   
  
 
 \bibitem[Siu87]{Siu87} Y.-T.Siu, 
  {\it Lectures on Hermitian-Einstein metrics for stable bundles and K\"ahler-Einstein metrics.}
   DMV Seminar, 8. Birkhäuser Verlag, Basel, 1987. 171 pp.
   

   
\bibitem[ST17]{ST17} J.Song, G.Tian,   
\emph{The {K}\"ahler-{R}icci flow through  singularities}, 
Invent. Math. {207} (2017),  519--595.
  

 
\bibitem[SzTo11]{SzTo11} G.Sz\'ekelyhidi, V.Tosatti, 
{\em Regularity of weak solutions of a complex Monge-Amp\`ere equation}, Anal. PDE {\bf 4} (2011), no. 3, 36


 \bibitem[Tian97]{Tian97} G.Tian,
 {\em K\"ahler-Einstein metrics with positive scalar curvature}, Inv. Math. 130 (1997), 239--265.
 
\bibitem[TY86]{TY86} G. Tian and S.-T. Yau, {\em Existence of K\"ahler-Einstein metrics on complete K\"ahler manifolds and their applications to algebraic geometry}, in {\it Mathematical aspects of string theory (San Diego, Calif., 1986)}, 574--628, Adv. Ser. Math. Phys., 1, World Sci. Publishing, Singapore. 


 \end{thebibliography}
\end{document}